\documentclass[]{amsart}
\usepackage{amsmath}
\usepackage{amsfonts}
\usepackage{amssymb}
\usepackage{amsthm}
\usepackage{multirow}

\usepackage{tikz}
\usepackage{mathabx} 
\usepackage{geometry}
\usepackage{setspace}
\usepackage{enumitem}

\usepackage[T1]{fontenc}
\usepackage[utf8]{inputenc}
\usepackage{float}

\newtheorem{theorem}{Theorem}[section]
\newtheorem{proposition}[theorem]{Proposition}
\newtheorem{lemma}[theorem]{Lemma}
\newtheorem{defn}[theorem]{Definition}
\newtheorem{cll}[theorem]{Corollary}

\theoremstyle{remark}
\newtheorem{remark}[theorem]{Remark}

\title{Phase transitions of McKean-Vlasov SDEs in Multi-well Landscapes}
\author{Alexander Alecio}

\date{\today}

\begin{document}

	\begin{abstract}
		Phase transitions and critical behaviour of a class of MV-SDEs, whose concomitant non-local Fokker-Planck equation includes the Granular Media equation with quadratic interaction potential as a special case, is studied. 
		By careful analysis of an implicit auxiliary integral equation, it is shown for a wide class of potentials that below a certain `critical threshold' there are exactly as many stationary measures as extrema of the potential, while above another the stationary measure is unique, and consequently phase transition(s) between. For symmetric bistable potentials, these critical thresholds are proven to be equal and a strictly increasing function of the aggregation parameter. Additionally, a simple condition is provided for symmetric multi-well potentials with an arbitrary number of extrema to demonstrate analogous behaviour. This answers, with considerably more generality, a conjecture of Tugaut [Stochastics, 86:2, 257-284].
		
		To the best of our knowledge many of these results are novel. Others simplify the proofs of known results whilst greatly increasing their applicability.


	\end{abstract}

	\keywords{McKean Vlasov diffusions; phase transitions; invariant probabilities; interacting particles systems}
	\subjclass{60H10, 60G10, 60J60, 82C22, 35Q82, 35Q83}

	\maketitle	

The general form of McKean-Vlasov SDEs (MV-SDEs) \cite{mckean} in one dimension is
\[
dX_t=b(X_t,\mu_t)dt+\sigma(X_t,\mu_t)dW_t, \qquad\mu_t=\text{Law}(X_t),\, t>0
\]

In this work we focus on one-dimensional MV-SDEs with separable drifts that depend on $\mu_t$ via an integral functional, $b_\theta(x,\mu)=f_1(x)+\theta\mathbb{E}_{\mu}(f_2(x))$ and diffusion $\sigma(x)=\sigma k(x)$, where aggregation and diffusion strength parameters ($\theta$, $\sigma$ respectively) are strictly positive constants. Concretely, this is
\begin{equation}\label{proto}
	dX_t=\big(-V^{'}(X_t)-\theta(P^{'}(X_t)-\mathbb{E}_{\rho_t}[P^{'}(x)]\big)dt+\sigma k(X_t)dB_t,\qquad\rho_t=\mathrm{Law}(X_t),\, t>0
\end{equation}
where $k>\epsilon$. MV-SDEs of this form have found numerous applications, of which we mention systemic risk \cite{ss} and global optimisation \cite{go}.

The concomitant Fokker-Planck Equation of (\ref{proto}) is of non-linear, non-local type,
\begin{equation}\label{fpe}
	\frac{\partial}{\partial t}\rho=\frac{\partial}{\partial x}\Big((-V^{'}(x)-\theta(P^{'}(x)-\mathbb{E}_{\rho}[P^{'}(x)]))\rho+\frac{\sigma}{2}\frac{\partial}{\partial x}k^2(x)\rho\Big)
\end{equation}
The particular form of the drift with respect to aggregation parameter is such that, when $P^{'}=x$ and $k=1$, the granular media equation is recovered
\begin{equation}\label{gme}\frac{\partial}{\partial t}\rho=\frac{\partial}{\partial x}(\rho\frac{\partial}{\partial x}(\log(\rho)+V+F*\rho))\end{equation}
for $F=\frac{1}{2}x^2$, \cite{CGM}. Although unused in this work, it noteworthy that (\ref{gme}) is a gradient flow, with respect to the Wasserstein metric, of the free energy functional
\begin{equation}\label{gf}
\mathcal{F}[\rho]=\sigma^2\int\rho\ln\rho\,dx+\int V\rho\,dx +\iint F(x-y)\rho(x)\rho(y)\,dx dy 
\end{equation}
for constant $k$; \cite{tam1,tam2}.

Customarily, solutions of MV-SDE (\ref{proto}) are derived as the hydrodynamic limit of a system of $N$-SDEs driven by independent Wiener processes, with the same diffusion and drift as (\ref{proto}) but with their empirical measure replacing $\mu$: 
$$
dX^{i,N}_t=b(X^{i,N}_t,\tilde\mu_t)dt+\sigma(X^{i,N}_t,\tilde\mu_t)dW_t^{i}, \qquad\tilde\mu_t=\sum_{j}^{N}\delta_{X_t^{j,N}},\, t>0
$$
This is the `Propogation of Chaos', which is described by a number of papers with differing conditions, for instance \cite{leo,sznit}. 

Existence and uniqueness of solutions to the Fokker-Planck equation (\ref{fpe}) is more straightforward to demonstrate, where positivity of solutions yield a priori estimates on functional $\mathbb{E}(P^{'})$.  At stationarity,  direct integration of (\ref{fpe}) combined with no-flux boundary conditions yields an equivalent integral formulation, provided $V^{'},\,P^{'},\, k \in C(\mathbb{R})$, entirely analogously with the linear stationary Fokker-Planck equation \cite{stroock}

In turn, this can be solved for the general form of the stationary measure
\begin{equation}\label{stq}
	\rho(\sigma,x,m)=\exp\Big(-\frac{2}{\sigma^2}\int^x\frac{V^{'}+\sigma^2kk^{'}+\theta(P^{'}(x)-m)}{k^2}\Big)
\end{equation}
(assuming such $V^{'},\,P^{'},\, k^{'}$ such that $\rho\in L^1(\mathbb{R})$) with the important caveat that $\rho(m)$ is an admissible stationary measure if and only if $m$ is a solution to the auxiliary equation 
$m=\int-V^{'}\rho(m)dx$, better known as the self-consistency function.

There are well known examples of MV-SDE (\ref{proto}) (\cite{dawson,shiino} for instance), whose self-consistency equation can have a single or multiple solutions depending on parameters choice, 
thus multiple stationary measures and a richness of long time behaviour. Casting $\sigma$ as the control parameter, admissible 
stationary solutions can be viewed as phases, whose characteristic property $\mathbb{E}_{\rho_0(m)}(P^{'}(x))$ (solutions of the 
self-consistency function) plays the  r\^ole of order parameter. Bifurcations of the order parameter as a function of the control 
parameter are then continuous phase transitions. For a comprehensive tutorial introduction, with discussion of compatiblity with Landau theory and connexions to self-organisation and synergetics, consult 
\cite{frank}.



The purpose of this work is to better understand this critical behaviour which, given much of the attraction of MV-SDE models is in this critical behaviour and multitude of long time dynamics, would seem a timely contribution. Particularly, we study the number of stationary solutions and their location, along with phase transition points (or critical transitions) and their dependence on aggregation parameter.

\subsection{Set-up}\label{sec:setup}
While it is clear that these questions can be approached via a study of the self-consistency equation, our results take advantage of a useful equivalence to the first moment equation, which becomes the central object of study in this paper. 

The genesis of this approach arose from numerical studies simulating SDE (\ref{proto}) without a free energy functional (\ref{gf}). Earlier simulation approaches would expand $P^{'},\,V^{'}$ as a truncated power series $a^{(n)},\,b^{(n)}$ and recast the SDE as a denumerable system of ODEs, 
\[
\dot m_i=\mathbb{E}(dX^i_t)=i\mathbb{E}[ a^{(n)}(X_t)X_t^{i-1}dt]+\sigma\frac{i(i-1)}{2}\mathbb{E}[b^{(n)}(X_t)X_t^{i-2}dt]
\]
allied with a method of closure; \cite{agp,bell}.

Even the simplest of these methods were surprisingly effective in representing critical behaviour with only a few retained moments. This was surprising, as it is not immediate in what manner the self-consistency equation manifested in the moment evolution equations. That the number of retained moments could be taken so low suggested it was the lowest order moment equation that encapsulated the self-consistency equation.

This motivated further work, eventually showing the first moment evolution equation and self-consistency equation were equivalent. This has interesting ramifications for the applicability of moment truncation schemes to SDE (\ref{proto}) and, by extension cumulant truncations; \cite{alecio2}

In order to proceed, we formalise these remarks, and those made in the introduction, and introduce definitions needed throughout the rest of this work.
\begin{defn}(Self-Consistency Function)
	The self-consistency function is 
\begin{equation}\label{sfcn}
	G_{(\sigma,\theta)}(m)=\int (P^{'}(x)-m)\rho(\sigma,x,m) dx 
\end{equation}
Solutions of the equation $F_\sigma(m)=0$ correspond to admissible stationary measures of (\ref{proto}) 
\end{defn}

The following term has become established in related literature, highlighting parallels between (\ref{statmeas}) and a stationary measure of a Smoluchowski SDE:
\begin{defn}(The Effective Potential)	
$\bar V_{\theta}(x,m)=\int\frac{V^{'} + \theta( P^{'}-m)}{k^2}$, $\theta\in\mathbb{R}^+$
\end{defn}
We will also refer to the potential $V$, although in this work we actually take its derivative $V^{'}$ rather than the potential as the fundamental object of study. Assumptions on $V$ will be given in the relevant sections.

It is useful, in Section \ref{sec:mw1} in particular, to distinguish the exponential `Gibbs' measure' part of the stationary measure
\begin{equation}\label{statmeas}
	\rho(\sigma,x,m)=\frac{1}{k^2}\exp\Big(-\frac{2}{\sigma^2}\bar V_{\theta}(x,m)\Big):=\frac{1}{k^2}\rho_0(\sigma,x,m)
\end{equation}
\begin{proposition} The self-consistency equation is equivalent to to the first moment evolution equation
\begin{equation}
	\label{rfsfcn}
	F_{(\sigma,\theta)}(m) = \frac{1}{\theta}\int -\frac{V^{'}}{k^2}\rho_0(\sigma,x,m) dx
\end{equation}
\end{proposition}

\begin{proof} Working from Definition \ref{sfcn}

\begin{align} G_\sigma(m)=\int (P^{'}(x)-m)\rho(\sigma,x,m) &dx = \int (P^{'}(x)-m)\frac{1}{k^2}\exp(\bar V_\theta(x,m)) dx
	\\&=\int \partial_x\big(\exp(-\frac{2}{\sigma^2}\int\frac{P^{'}-m}{k^2})\big)k^2\exp(-\frac{2}{\sigma^2}\int\frac{V^{'}}{k^2})dx
\end{align}

Assuming sufficient regularity to ignore the boundary terms, the result follows from an integration by parts.
\end{proof}

Working with the first moment evolution equation leads to robust results with intuitive proofs without much technical obfuscation. Further they can yield quantitative estimates of parameters if required. Motivated by applications, we refrain from overly prescriptive assumptions on $V^{'}$ to maintain flexibility.

Throughout this work $J$ always denotes a connected compact set. Conditions introduced in each section will always ensure $\rho$ is normalisable and $F_\sigma$ exists for all $m,\sigma,\theta$ in their respective domains. We take $\rho$ normalised in Section \ref{sec:mw1} and unnormalised in Section \ref{sec:secsym}. We will frequently suppress $\theta$ dependence when context allows.

In terms of $F_\sigma$, a critical transition is any value of $\sigma$ where the number of roots changes. Assuming they exist, the upper (lower) critical transition is the largest (smallest) such value.

\subsection{Outline and relation to other works}
This paper naturally divides into two halves, the first deals in problems with fairly general conditions, while the second involves results that require some symmetry.

The main results are:
\begin{itemize}
\item For suitably smooth multi-well potentials with unimodal $\rho$ (specifically that $\bar V^{'}$ is a diffeomorphism, subsequently weakened to a homeomorphism).
\begin{itemize}

\item Below the lower critical threshold, $\sigma<\sigma_c^l$, there are exactly as many stationary measures as simple roots of $\bar V^{'}$ - Proposition \ref{multiwell} and Corollary \ref{at2}

\item Above the upper critical threshold, $\sigma_c^u<\sigma$, there is only one stationary measure - Proposition \ref{upb}

\item Multimodal $\rho$ is considered as a special case and a counter-example given where the number of stationary measures is less than the number of roots of $V^{'}$ - Remark \ref{multim}
\end{itemize}

\item For antisymmetric $V^{'},\,P^{'}, k$, but otherwise looser assumptions,
\begin{itemize}
\item With a symmetric bistable potential $\sigma_c^l=\sigma_c^u$ and $\sigma_c^u(\theta)$ is an increasing function - Proposition \ref{psd} and \ref{ctgb}.

\item We define a class of symmetric multi-well potentials that behave similarly to the bistable potential, particularly that $\sigma_c^u(\theta)$ remains an increasing function - Proposition \ref{vainilla} and Corollary \ref{c2fg}
\end{itemize}
\end{itemize}

There has been much sustained activity with regard to the convergence to stationary measures of MV-SDEs, \cite{rocky,eberle}. Recently, there has been a growing literature based on modern variational methods for the granular media equation such as \cite{tug2}, the progenitor of which appears to be \cite{shiino}, which variationally studies convergence for the symmetric polynomial bistable problem.

Exact studies of the properties of these stationary solutions has been scarcer. For the symmetric polynomial bistable problem, a result substantially like Proposition \ref{psd} is given, through a study of the self-consistency equation by \cite{shiino}. The proof relies heavily on the the GHS inequality, which is resistant to generalisation. For a similar, entirely non-variational study, see \cite{dawson}

In a comprehensive series of papers, author Tugaut has studied many problems related to those considered here. 
The closest in scope and applicability to this paper is \cite{tug}, which employs a self-consistency equation centered method to study polynomial bistable potentials and positive interactions of the form $x^{2n}*\rho,\,n\geq2$.  The first half of \cite{tug} is dedicated to quadratic interactions, corresponding to the Granular Media equation with interaction kernel $F=\frac{x^2}{2}$ which, as already discussed, our interactions incorporate. In this overlap our results strengthen his with more general assumptions, in the process simplifying the proofs. For instance, Proposition 1.2/ Corollary 1.3 of \cite{tug} demonstrates that, for sufficiently small $\sigma$, there are at least as many stationary measures as roots $V^{'}$ which is improved to equality by our Proposition \ref{multiwell}. Theorem 2.1 of \cite{tug} establishes there are exactly three solutions for bistable potentials of the form $\sum^N_{j=2} |a_j|x^{2j}-x^2$. Our Proposition \ref{psd} simplifies the argumentation and extends the result to a large class of bistable potential. Further, for such polynomial bistable potentials, on p.270 it is conjectured $\sigma^u_c(\theta)$ is increasing, which we answer in Proposition \ref{ctgb} and extend to symmetric multi-well potentials in Proposition \ref{vainilla}. 

\section{Generic Multi-well Potentials}\label{sec:mw1}
This section is devoted to MV-SDE (\ref{proto}) with a large class of polynomial bounded multi-well potentials and strictly convex (unimodal)effective potentials, subsequently loosened to convex. The non-convex (or multimodal) problem is considered as a series of convex ones.

\subsubsection{Definitions and Assumptions}
We introduce the following definitions for this section. 
First is the inverse of a function that maps $m$ to the mode of (\ref{statmeas})
\begin{defn}(Modal dependence) 
$x^{*-1}=\frac{1}{\theta}(V^{'}+\theta P^{'})$.
\end{defn}
The second is the convenient short-hand $a=\int_0^x\frac{1}{k^2}$.

Next, we introduce a series of conditions for the results of this section.
They are phrased in terms of $V^{'}$ and $\bar V$, which was seen to be their most natural form given the proof of Proposition \ref{multiwell} and others.
\begin{enumerate}
	\item $V^{'}(x,0),\,P^{'}\in C^2(\mathbb{R})$, $k\in C^1(\mathbb{R})$.\label{a1}
	\item \(\lVert a^{'}\rVert_{\infty}=\lVert \frac{1}{k^2} \rVert_{\infty}<\infty\) \label{a3}
	\item \(\lim\limits_{|x|\rightarrow\infty}\frac {\bar V(x,0)}{x^2}>0\) \label{a4}
	\item \(\lim\limits_{x\rightarrow\pm\infty}\bar V^{'}(x,0)=\lim\limits_{x\rightarrow\pm\infty} V^{'}=\pm\infty \)\label{a2}
 	\item $|V^{'}|<K(1+x^{2N})$ \label{a5}
	
	\item $(x^{*-1})^{'}>0$ \label{v2d}
\end{enumerate}

Assumptions \ref{a1}, \ref{a4}, \ref{a2} \& \ref{a5} provide useful bounds ensuring, amongst other things that $F_\sigma(m)$ exists for all $m$, there is than sufficient regularity for the integration by parts and the boundary terms are null. Assumption \ref{a2} ensures $F_\sigma(m)$ is bounded away from zero far away from the origin. It is possible to weaken the last condition, which we consider as special cases of the above. See Corollary \ref{at2} and the Remark \ref{multim}.


\begin{proposition}[Convergence of $F_\sigma(m)$ (\ref{sfcn})]
	\label{multiwell}
	
	Under Assumptions (\ref{a1})-(\ref{v2d}), $x^{*}$ exists and, as $\sigma\downarrow 0$, $F_\sigma(m)$ converges to $-\frac{V^{'}}{k^2}\circ x^{*}(m)$ uniformly on compact sets.
	 
	Moreover, let $J$ denote some connected compact set containing $N<\infty$ zeros of $V^{'}$ strictly in its interior. If all these zeros are simple, then there exists a critical threshold $\sigma_c$ such that for $\sigma<\sigma_c$, $F_\sigma(m)|_{x^{*-1}(J)}$ has precisely $N$ zeros.
\end{proposition}

\begin{proof}	   
	That $x^{*}(m)$ exists and its range is $\mathbb{R}$ is an immediate consequence of Condition \ref{a2} and \ref{v2d}. 

	By Laplace's Method,
	\begin{equation}
		\label{lp}
		\lim\limits_{\sigma\downarrow 0}\int_{\mathbb{R}}f(x)\frac{\exp(-\frac{2}{\sigma^2}(\bar V + m a))}{\sqrt{\frac{\pi}{\bar V^{''}(x^*)}}\exp(-\frac{2}{\sigma^2}\bar V(x^*))}=f(x^*)	
	\end{equation}
	where we have suppressed $m$ dependence in $x^*$.

	On compact sets (of $m$), it follows from our assumptions that $V^{'},\,P^{'},\, \lVert\frac{V^{'}}{k^2}\exp(-(\int\frac{\bar V - \theta ma}{k^2})\rVert_{1}$ are all bounded and $\bar V^{''}(x^*)=\frac{(x^{*-1})^{'}}{k^2}(x^*)>\epsilon>0$. The standard proof of the Laplace method shows that convergence is uniform on compact sets when $f=-\frac{V^{'}(x)}{k^2}$\footnote{Strictly speaking, a constant may need to be added to $-V^{'}$ to ensure it is of one sign on the chosen compact set before applying the Laplace method. This is then subtracted off the limit to get the result} or a constant. 
	
	As limit (\ref{lp}) with $f=1$ is bounded away from 0, the limit of the reciprocal exists and must be uniformly convergent also. $\lim\limits_{\sigma\downarrow 0}F_\sigma(m)$ is the product of this reciprocal and (\ref{lp}) with $-\frac{V^{'}(x)}{k^2}$, it is also uniformly converging too with the limit $-\frac{V^{'}}{k^2}\circ x^{*}$, given the denominators from (\ref{lp}) cancel each other. 
	

	By the uniform convergence of $F_\sigma$ and the intermediate value theorem,  there must be at least one root in a neighbourhood of $x_0$, a simple root of $-V^{'}$\footnote{To be precise, in the neighbourhood of $m_0$ where $x^*(m_0)=x_0$} which can be made arbitrarily small with $\sigma$.
	To establish the second part of the Proposition, it's sufficient to show $\frac{\partial F_\sigma}{\partial m}$ has the sign of $-V^{''}(x_0)$ in a fixed interval centered on $x_0$, for all $\sigma$ arbitrarily small. Given there must be at least one root in the neighbourhood by the intermediate value theorem, and that $F_\sigma$  is there strictly increasing (respectively decreasing), we can conclude there is one root for sufficiently small $\sigma$.


	Concretely, we will establish there is a some compact interval $A$ on which $\frac{\partial F_\sigma}{\partial m}$, has the sign of $-V^{''}(x_0)$, for all $x^{*}\in A$.
	The crux of the method is to identify and exploit the covariance structure of $\frac{\partial}{\partial m}F_\sigma$:
	\[ \frac{\partial}{\partial m}F_\sigma=\frac{2}{\sigma^2}\Big(\int_{\mathbb{R}} -aV^{'}\rho dx -\int_{\mathbb{R}} a\rho dx\int_{\mathbb{R}} -V^{'} \rho dx\Big)=\frac{2}{\sigma^2}\mathrm{Cov}_{x^*}(a,-V^{'})\] 
	with $a=\int^x\frac{1}{k^2}dx$. Particularly useful is the alternate form $\mathrm{Cov}_m(a,-V)$,
	\begin{equation}\label{covi}\iint_{\mathbb{R}^2}\big(a(x)-a(y)\big)\big(-V^{'}(x)-(-V^{'}(y))\big)\frac{\rho_0(x,x^{*})}{k^2(x)}\frac{\rho_0(y,x^{*})}{k^2(y)}dxdy \end{equation} where we have used (\ref{statmeas}). 
	Assuming $x,y$ are sufficiently close to $x_0$ that $V^{'}(x)$ is suitably well approximated by its Taylor expansion, and recalling $a(x)$ is strictly increasing, it can be seen that the integrand $(a(x)-a(y))(-V^{'}(x)-(-V^{'}(y)))$ is positive or negative depending on the sign of $-V^{''}(x_0)$, by taking the cases $x<y$, $y<x$. 
	By bounding the integrand, we will demonstrate rigorously that $\rho$ weights this region sufficiently that the integral has the same sign.

	Without loss of generality, we assume $-V^{''}(x_{0})\geq 0$.
	As before, we replace $\mathcal{Z}$ with $$\mathcal{\bar Z}=\sqrt{\pi}\sigma \exp(-\frac{2}{\sigma^2}\bar V(x^{*}))$$ 
	incorporating $\bar V(x^{*})$ from $\mathcal{\bar Z}$ into the exponent part of $\rho$.
	We proceed by splitting the domain of integration into two parts, $R_1(x^*)=B_R{(x^*,\,x^*)}$ for $x^*$ in some set $A$, and its complement $R_2$.
	
	To define these regions, we consider the following factors. By Assumption \ref{a1} there exists an interval around root $x_0$, $$I=[x_0-2\delta,x_0+2\delta]$$ on which $-V^{''}>-\underline V^{''}>0$. Further on $I$ we can bound $a^{'}=\frac{1}{k^2}$ below by $\frac{1}{\underline k^2}>0$ by Assumption (\ref{a1}) and the Extremal Value theorem applied to $k$. Then with the mean value theorem, 
	\[
	(a(x)-a(y))(-V^{'}(x)-(-V^{'}(y))) \geq -\underline V^{''} \frac{1}{\underline k^2} (x-y)^2>0
	\]
	for $x,y\in I\times I$. 
	
	 We bound $\bar V(x)- \bar V(x^*)$ above by $\alpha(x-x^*)^2,\, \alpha>0$ for $[x,x^*]\subset I$ using the Taylor remainder theorem and Assumption \ref{a1}. 

	Consequently we identify 
	\[
	\begin{split}
	& A=[x_0-\delta,x_0+\delta] \\
	& R(x^*)=\bar B_\delta(x^*,x^*),\qquad R=\delta, \quad x^*\in A
	\end{split}
	\]
	By construction, $R(x^*)\subset I\times I$ when $x^*\in A$, so the previously established bounds are still applicable.  
	
	We bound ($\ref{covi}$) below on $R_1$ by 
	\[\frac{2}{\sigma^2}\iint_{R_1}\frac{-\underline{V}^{''}}{\underline k^2}\big((x-x^*)-(y-y^*)\big)^2\frac{\exp\Big(-\frac{2\alpha}{\sigma^2}((x-x^*)^2+(y-x^*)^2)\Big)}{\underline k^4 2\pi\sigma^2}dxdy\]
	where we have dropped the dependency of $R$ on $x^*$, as our bounds are uniform in $x^*$. Transforming to polar coordinates - $r\cos(\theta)=(x-x^*),\, r\sin(\theta)=(y-x^*)$ - we have 
	\[
		\frac{2}{2\pi\sigma^4}\int_0^{2\pi}(1-\sin(2\theta))d\theta\,\frac{-\underline V^{''}}{\underline k^6} \int_{R_1} r^{2+1}\exp(-\frac{2}{\sigma^2}\alpha r^2) dr \quad \propto \quad \frac{1}{\sigma^4}(\sigma^4-\exp(-\frac{2}{\sigma^2})(\dots))
	\]
	where (\dots) is polynomial in $\frac{1}{\sigma^2}$. Consequently the second, negative, term can be made arbitrarily small.
	
	On $R_2$, our argument is similar. We can bound the integrand  $|(V^{'}(x)-V^{'}(y))(a(x)-a(y))|<2K(2+x^{2N}+y^{2N})|x|$ by Assumption \ref{a3} and \ref{a5}. 
	Moreover, independently of bounded $x^*$, as $\bar V-\bar V(x^*)$ is (super-) quadratic outside some finite radius (Assumption \ref{a4}) and $x^*$ is the sole minimum, we can bound $\bar V(x)-\bar V(x^*)>\beta(x-x^*)^2,\, \beta>0, \, x^*\in A$. 
	
Putting these bounds together and transferring to polar coordinates once again, we have the following lower bound on $R_2$ 
	
	\[
	-K\int \frac{r^{2N+2}}{\bar k^6}\exp(-\frac{2\beta}{\sigma^2} r^2)dr
	\]
which decays $\exp(-\frac{2}{\sigma^2})$. Adding these two bounds we see the integral has the sign of $-V^{''}(x_0)$ for all $x^{*}\in A$, and the result follows. 
\end{proof}
Assumption 6 is by far the most onerous restriction needed for Proposition \ref{multiwell}, requiring $\bar V^{'}$ to be a diffeomorphism. This can be loosened to a homeomorphism with the following assumptions.
\begin{enumerate}
	\item[(7)] $\bar{V}^{''}(x,0)\geq 0$\label{v3d}, where the lower bound is attained at a finite number of isolated points $\{\tilde{x}_i\}_i^n$ which are still global maxima of $\rho$  when $x^{*}=\tilde{x}$
	\newline and
	\item[(8)] $V^{''}(\tilde{x},0)\neq 0$, $\forall\tilde{x} \in \{\tilde{x}_i\}^n_i $ \label{v4d}
\end{enumerate}
\begin{cll}\label{at2}
	With Assumption 6 replaced by 7 \& 8, Proposition \ref{multiwell} still holds in its entirety. 
\end{cll}

\begin{proof}
If Assumption 7 holds, Proposition \ref{multiwell} can be applied to all but a finite number of open intervals containing a $\tilde{x}$ which can be made arbitrarily small with $\sigma$. It remains to prove the result in these intervals which as it happens, save for establishing uniform convergence, can be handled with an argument substantially similar to Proposition \ref{multiwell}. We outline this, carefully noting any divergences, returning to the question of uniform convergence at the end. 

As $\tilde{x}$ is the unique maxima of $\rho(x,\tilde{x})$ , it is still possible to define $x^{*}$, which is continuous. Moreover, the Laplace method can be adapted to the lowest order non-null derivative which must be of even order and positive. Applying the method to both the integral and $\mathcal{Z}$ in the same way as the first part of the above formulation, we see the limit is as described.

The second part concerning bounding the derivative at roots of $V^{'}$ away from 0 is almost entirely applicable, except we caveat that the bound $\bar V - \bar V(x^*)>\beta(x-x^*)^2$ only holds outside an arbitrarily small interval around $x^*$. As this bound need only hold on $R_2$, this is not an obstacle.\footnote{The derivative at $\tilde{x}$ may blow up as product of $\sigma^{-2}$ and an integral with a non-zero $\sigma^{2-\delta}$ term. As $F_\sigma$ is totally bounded this does not affect the proof}

To show convergence is uniform, we invoke Assumption 8. 
With this, we know $V^{''}(x)$ is of one sign in some closed interval centered on $\tilde{x}$. To that interval, we apply the method of the proof of the second part of Proposition \ref{multiwell} to demonstrate the derivative cannot be null (with the same caveats as described on the last paragraph), and conclude $F_\sigma$ is the limit of strictly/increasing decreasing functions on this region. As $F_\sigma$ is bounded on such an interval, convergence is uniform by corollary of the Helly Selection theorem, Exercise 7.13 (b) of \cite{rud}.
As there are at most a finite number of such intervals, we conclude convergence is uniform on $J$.
\end{proof}

We furnish these results with a few remarks and examples. 

\begin{remark}\label{exs}[Examples]
Consider the simplest polynomial bistable potential $V^{'}=x^3-x$ with quadratic interaction $P^{'}=x$. Then $\bar V^{'}=x^3+(\theta-1)x$. There are three cases $\theta<1,\theta=1,\theta>1$. 

\begin{itemize}
\item[$\centerdot$] $\theta>1$, $\bar V^{'}$ is strictly increasing, so Proposition \ref{multiwell} can be applied.

\item[$\centerdot$] $\theta < 1$, $\bar V^{'}$ is not strictly increasing and cannot be tackled with neither Proposition \ref{multiwell} of Corollary \ref{at2}, though it can be with the results of section \ref{sec:bistable}

\item[$\centerdot$] $\theta=1$ is the transition between the two regimes, which manifests as a point of inflexion at 0. However $-V^{''}(0)=1$ so Corollary \ref{at2} can be applied.
\end{itemize}
\cite{alecio3} considers the related problem of $\bar V^{'}=\frac{x^3+(\theta-1)x}{1+x^2}$, where the same parameter regime holds.
\end{remark}

\begin{remark}{(Simple zeros)}
	To understand the restriction to simple zeros, consider $V^{'}=x^3$, $P^{'}=x^2$. $V^{'}$ has a double root at 0, but for sufficiently small $\sigma $, $F_\sigma>0$ near zero (see Remark \ref{safety}). Only at $\sigma=0$ can there be a zero of $F_\sigma$ at $0=\bar x^{*-1}(0)$. 
\end{remark}

\begin{remark}{(Multimodal $\rho$)}
	\label{multim}
	In requiring $\bar V^{''}>0$, Proposition \ref{multiwell} is limited to unimodal $\rho$. This can be generalised to $\rho$ possessing a finite number of modes, where Laplace's method applies to the largest of them.
	
	 The general form of $\rho$ in (\ref{statmeas}) means $\rho$ can still parameterised by $m$ as $x^{*-1}(m):=\arg\max_{x} \rho (m,x)$ is an increasing, piecewise continuous function. Discontinuities correspond to multiple modes of equal height, where it is still possible to define the limit $F_\sigma$. For example, the bistable potential, $P^{'}=x$ with $\theta<1$, $\bar x^{*-1}(m)$ will be discontinuous at $m=0$ but $F_\sigma(0)$ is the average of the left and right limit, which is itself zero.
	
	Discontinuities may restrict the zeros of $F_\sigma$. 
	Consider $V^{'}=x(x^2-1)(x^2-4)$, $P^{'}=x$ and $\theta=1$. $V^{'}$ has a root at 1, but $x^{*}(m)$ cannot be unity because of discontinuity.
	
\end{remark}

For the sequel, we note that if $P^{'}$ is increasing, any $\bar V$ will eventually be unimodal for sufficiently large $\theta$ as $|V^{'}|$  is eventually unbounded, Assumption \ref{a2}
\begin{lemma}\label{eventmono}
	Suppose $P^{''}$ has a positive lower bound. Then there exists $\theta^*$, such that for all $\theta>\theta^*$ $\bar V_\theta^{'}(x,0)$ is strictly increasing. 
\end{lemma}

It is considerably easier to demonstrate the existence of $\sigma_c^u$ than $\sigma_c^l$. The following Proposition uses a very straightforward bounding argument with no additional assumptions to Proposition \ref{multiwell}

\begin{proposition}\label{upb}
	There exists an Upper Critical threshold $\sigma_c^u$, such that for $\sigma>\sigma_c^u$, $F_\sigma(m)$ has only one root.
\end{proposition}

\begin{proof}
	If it could be shown that above some $\sigma_c^u$ $\frac{\partial F}{\partial m}(m)<-\epsilon<0$, then it is clear that $F(m)$ could be bounded above anywhere by a decreasing linear function. Combined with the facts that $F(m)$ exists for all finite $m$, and $\lim_{m\rightarrow\pm \infty}F(m)=\mp\infty$, we could conclude there exists a root, which must be unique.

	It remains to establish the inequality on $\frac{\partial F}{\partial m}(m)$. As an inequality on the sign, we work with the unnormalised stationary measure, 
	\[
		\frac{\partial F}{\partial m}=\int_{\mathbb{R}}a(-V^{'})\rho(\sigma,x,m) dx    
	\]

    By Assumption \ref{a3} and \ref{a2} there exists an $R$ such that $aV^{'}>c$, on $\mathbb{R}\backslash B_0(R)$. 
	To control the integral is split over an inner region, $B_0(R)$ and outer domain, $\mathbb{R}\backslash B_0(R)$. Setting $\mathcal{Z}=\exp\big(-\frac{2}{\sigma^2}\bar{V}(x^*(0))\big)$ and repeatedly using $a$ is increasing with $a(-R)<0<a(R)$:

    \begin{align}
        \int_{-R}^{R}a(-V^{'})\frac{\mathcal{Z}}{\mathcal{Z}}\rho(\sigma,x,m) dx&<C\mathcal{Z}\int_{-R}^{R}\frac{\rho(\sigma,x,0)}{\mathcal{Z}}\exp(\frac{2}{\sigma^2}am)dx\\
        \label{luub}&<C_1\mathcal{Z}\big(\mathbb{I}_{m\geq0}\exp(\frac{2a(R)}{\sigma^2}m)+\mathbb{I}_{m<0}\exp(\frac{2a(-R)}{\sigma^2}m)\big)
    \end{align}
    with $C=\sup_{[-R,R]}\frac{|a(-V^{'})|}{k^2}$ and $C_1=2RC$. Similarly, 
    \begin{align}
        \mathcal{Z}\int_{\mathbb{R}\backslash B_0(R)}aV^{'}&\frac{\rho_0(\sigma,x,m)}{\mathcal{Z}} dx>C_2\mathcal{Z}\int_{R}^{\infty}\frac{\rho(\sigma,x,0)}{\mathcal{Z}}\exp(\frac{2}{\sigma^2}am)dx\\
        \label{uuub}&>2C_2\mathcal{Z} I_\sigma\Big(\mathbb{I}_{m\geq0}\exp(\frac{2a(R)}{\sigma^2}m)+\mathbb{I}_{m<0}\exp(\frac{2a(-R)}{\sigma^2}m)\Big)
    \end{align}
    where $C_2=c\cdot\mathrm{min}(-a(-R),a(R))$ and $I_\sigma=2\int_R^\infty\frac{\rho(\sigma,x,0)}{\mathcal{Z}} dx$

    With the monotone convergence theorem, $I_\sigma$ is an increasing unbounded function. Therefore it is possible to find $\sigma_c^{u}$ such that, for $\sigma>\sigma_c^{u}$, $I_\sigma>C_1+\epsilon$. Then, subtracting (\ref{uuub}) from (\ref{luub}), we have $\frac{\partial F}{\partial m}(m)<-\epsilon<0$. The claim is therefore proven and the desired result established.
	
\end{proof}
\begin{remark}
While this argument can be extended to multimodal $\rho$, we advance a different argument for the analogous Corollary \ref{unboundedbelow} which is idiomatic to symmetric potentials.
\end{remark}

As $F_\sigma$ is bounded away from zero outside some finite interval by Assumption \ref{a2}, we can count roots on all $\mathbb{R}$. We combine this with a restatement of Proposition \ref{multiwell} and \ref{upb}, with the same assumptions, in terms of stationary measures of MV-SDE(\ref{proto}). 
\begin{theorem}{Stationary Measures of MV-SDE (\ref{proto})}
\label{statm}
Suppose $V^{'}$ has $N<\infty$ zeros, $\{x_i\}_i^N$, all of which are simple. Then there exists critical transition thresholds $\sigma_c^l,\,\sigma_c^u$ such that $\sigma>\sigma_c^u$ MV-SDE (\ref{proto}) has one stationary measure while $\sigma<\sigma_c^l$, it possesses exactly $N$ stationary measures.
\end{theorem}
\begin{proof}
	Proposition \ref{upb} can be applied verbatim. 
	By Assumption \ref{a2}, there exists a finite interval $J$ such that $\{x_i\}^N_i \in J^{o}$ and $|\frac{V^{'}}{k^2}|>\delta, \, x\in J^c$.
	Applying Proposition \ref{multiwell} to $J$, noting $F_\sigma$ must be bounded away from 0 for sufficiently small $\sigma$, and the bijection between zeros and stationary measures, we can conclude the result. 
\end{proof}

\begin{remark}
	We are now in a position to compare MV-SDE (\ref{proto}) with $\theta>0$ to $\theta=0$, with assumptions 1-8. Below $\sigma_c^l$, every stationary measure is a unimodal distribution whose mean (and mode) can be made arbitrarily close to any extremal point (maxima and minima) of $V$. When $\theta=0$, there is one stationary measure whose modes coincide with the minima of $V$. Above $\sigma_c^u$, MV-SDE (\ref{proto}) becomes ergodic but is still unimodal.
\end{remark}

\begin{remark}[More Examples]
	Section 4 of \cite{gomes} introduces simple polynomial uni-stable ($V^{'}=x$) and bistable potential perturbed by separable and non-separable fluctuations, with quadratic interaction. 
	
	\cite{gomes} approaches the problem by defining the potential $V_0=\int_x\frac{V^{'}}{k^2}$ to which a perturbation is added, rather than working with its derivative as we have favoured here. In the case of separable perturbation, $\delta\cos(\frac{x}{\epsilon})$, Proposition \ref{multiwell} easily applies with $\theta>\frac{\delta}{\epsilon^2}$.

	The non-separable perturbation $\delta\mathbb{I}_{[-a,a]}x^2\cos(\frac{x}{\epsilon})$ requires a little more work as it is non-differentiable. Reading the remarks in \cite{gomes}, the indicator function is producted into the definition to control the fluctuations as $|x|\rightarrow\infty$ to later apply the homogenisation theorem, so $a$ should be read as `large'. In light of this, it is reasonable to only consider $a$ very close to a zero of $\cos(\frac{x}{\epsilon})$, and interpolate between the two with quadratic interpolants. This is in the unbounded region, so again Proposition \ref{multiwell} can be applied for suitably large $\theta$
\end{remark}

\section{Symmetric Effective Potentials}
\label{sec:secsym}
This section is concerned with SDE (\ref{proto}) with a symmetric effective potential (antisymmetric $P^{'}$ $-V^{'}$, symmetric $k$). Whilst a few results could be deduced from Section \ref{sec:mw1}, they would have narrower applicability and most, such as those on the Critical Transition and its dependence on $\theta$, cannot.

This is because, in return for this symmetry, some of the most onerous restrictions of Section \ref{sec:mw1} can be relaxed or even lifted, such as the prohibition on double roots in $-V^{'}$ and non-increasing $\bar V$ can be lifted. Amongst other interesting possibilities, this allows repulsive-attractive $P^{'}$.

\subsection{Bistable Potential}
\label{sec:bistable}
This section is concerned with SDE (\ref{proto}) with a antisymmetric $P^{'}$ and $-V^{'}$, which is bistable, i.e possessing three roots at $-x^*,0,x^*$. The polynomial bistable potential is well studied, \cite{dawson,desi,kome,shiino}.

The approach here is still rooted in a study of the self-consistency equation $F_\sigma$, or more precisely the equivalent first moment equation. After establishing some key propositions about the number of roots of $F_\sigma$, critical transition points are studied in detail.

The most basic assumptions are
\begin{enumerate}
	\item $-V^{'}$, $P^{'}$ and $k^{'}$ are antisymmetric and $C(\mathbb{R})$.
	\item $-V^{'}$ has three roots at \(\{0,\pm |x^*|\}\), and $-V^{'}(x)<0\, x>x^*$ 
	\item \(\lim\limits_{|x|\rightarrow\infty}\frac {\bar V(x,0)}{x^2}>0\)
	\item $|V^{'}|<K(1+x^{2N})$
\end{enumerate}

Assumption 1 and 3 ensures enough regularity for the integration by parts in section \ref{sec:setup} and that the boundary terms can be discarded. The polynomial bounds in the fourth ensure integrability and the second simply outlines a bistable potential. Further assumptions (5-8) will be introduced when the critical transition point is studied.

In much of this section, $\sigma$ dependence of $\rho$ is supressed. Parity restrictions (Assumption 1) are manifested in $F_\sigma$ according to the following lemma

\begin{lemma}\label{anti}
F(m) is antisymmetric
\end{lemma}

\begin{proof}
	Assumption 1 implies $\rho(x,m)=\rho(-x,-m)$, where $\rho(x,m)=\exp(-\frac{2}{\sigma^2}(\bar V(x,m)))$. Then
\[
	\begin{split}
	F(-m)=\int_{-\infty}^{\infty}-V^{'}(x)\rho(-m,x)dx\stackrel{y=-x}{=}-\int^{-\infty}_{\infty}(-V^{'}(y))\rho(-m,-y)dy=
	\\ \int_{-\infty}^{\infty}V^{'}(y)\rho(m,y)dy=-F(m)
\end{split}
\]
holding regardless of $\sigma$ or $\theta$
\end{proof}

The following proposition characterises the roots of $F_\sigma$. For constant diffusion, compare with (3.21) of \cite{shiino}, Theorem 3.31 of \cite{dawson} and, particularly, Theorem 2.1 of \cite{tug} 

\begin{proposition}[Properties of $ F(m)$ ]$ $
	\label{psd}
	\begin{enumerate}[label=(\roman*)]
		\item F has a root at 0.
		\item There is at most one strictly positive (negative) root of $F_\sigma$.
		\item These additional roots exist iff $\left.\frac{\partial F}{\partial m}\right\vert_{m=0}>0$
	\end{enumerate}
\end{proposition}

\begin{proof}
	Part (\textit{i}) follows immediately from the antisymmetry of $F_\sigma$, Lemma \ref{anti}.

	For Parts (\textit{ii}) \& (\textit{iii}), we derive the series representation of $F_\sigma$ by substituting $\exp(\frac{2}{\sigma^2}mx)$ with its series representation, 
	$$F(m)=\int-V^{'}(x)\rho(x,0)\sum_{n}\frac{(\frac{2}{\sigma^2}ma)^n}{n!}dx$$ 

	Given the antisymmetry of $F_\sigma$, we focus on $m\in\mathbb{R}^+$
	The strategy for the proof is to show the coefficients of the series are at most change sign once, from positive to negative. It is then straightforward to demonstrate that $F(m)$ must also have at most one root in $\mathbb{R}^{+}$
	
	With the polynomial growth conditions 3 \& 4, $F(m)$ can be bounded above as moments of a Gaussian distribution. Exchange of the summation and integral

	\begin{equation}\label{scps}
		F(m)=\sum_n\frac{(\frac{2}{\sigma^2}m)^n}{n!}\int-V^{'}a^n\rho(x,0)dx:=\sum_n\frac{(\frac{2}{\sigma^2}m)^n}{n!}I(n)
	\end{equation}
	can then be justified with an application of the Tonelli-Fubini Theorem.

	From the antisymmetry of the integrand, $I(2n)=0$. Reciprocally, by the symmetry of the integrand of $I(2n-1)$ 
	\begin{equation}
		I(2n-1)=2\int_{0}^{\infty}-V^{'}a^{2n-1}\rho(x,0)dx
	\end{equation}

	By assumption on $k$, $a$ must be strictly increasing. Consequently, 
	\begin{align}
		\begin{split}
	\big(\frac{a}{a(x^*)}\big)^{2n-1}<1,\,x<x^*\\
	\big(\frac{a}{a(x^*)}\big)^{2n-1}\geq 1,\,x\geq x^*
		\end{split}
	\end{align}
	
	With these facts, and introducing the rescaled coefficient
	\begin{equation}
	\tilde I(2n-1,[0,\infty))= 2\int_{0}^{\infty}-V^{'}\Big(\frac{a}{a(x^*)}\Big)^{2n-1}\rho(x,0)dx
	\end{equation}
	with 
	\[
		\frac{a}{a(x^*)}=\tilde a
	\]
	it can be seen both
	\begin{align}\label{ineq}
	\begin{split}
	0<\,&\tilde I(2n-3,[0,x^*])< \tilde I(2n-1,[0,x^*])\\
	&\tilde I(2n-3,[x^*,\infty))< \tilde I(2n-1,[x^*,\infty))<0
	\end{split}
	\end{align}
	by the monotonicity of the integral. Adding these two monotonically decreasing inequalities, we see $\tilde I(2n-1,[0,\infty))$ is also strictly decreasing, 
	\begin{equation}\label{monot}
		\tilde I(2n-1)<\tilde I(2n-3)
	\end{equation}
	
	Applying the Dominated convergence theorem,  $\lim\limits_{n\rightarrow\infty}\tilde I(2n-1,[0,x^*])\rightarrow 0$. Therefore there must exist a threshold $n_{t}$ such that $|\tilde I(2n_{t}-1,[0,x^*])|<|\tilde I(1,[x^*,\infty))|$ so, combined with monotonicity (\ref{monot}), we conclude $\tilde I(2n-1,[0,\infty))$ must eventually become negative.
	

	As a positive multiple of $\tilde I(2n-1, [0,\infty))$, $I(2n-1)$, while not necessarily monotonic, inherits the crucial property that there exists a threshold $n_c<\infty$ such that $I(2n-1)\geq 0$ iff $n\leq n_c$. 
	
	If all the $I(2n-1)$ are negative we set $n_c=0$, and clearly $F(m)<0, m>0$

	For $n_c\neq 0$, splitting the series by sign at $n_c$ and factorising
	\begin{equation}\label{splits}
		F(m)=m^{2n_c-1}\Big(\sum_1^{n_c}
		\frac{(\frac{2}{\sigma^2})^{2n-1}}{2n-1!}I_{2n-1}\frac{1}{m^{2(n_c-n)}} + \sum_{n_c+1}^\infty \frac{(\frac{2}{\sigma^2})^{2n-1}}{2n-1!} I_{2n-1}m^{2(n-n_c)}\Big)
	\end{equation}
	Inside the parenthesis is the difference of a strictly decreasing and increasing functions, where the limit as $m\downarrow 0$ is $\infty$ and $m\uparrow\infty$ is $-\infty$. We conclude $F_\sigma$ has at most one root in $\mathbb{R}^+$, where the sign must change from positive to negative.

	As has just been outlined, a root can only occur if $n_c\neq 0$. By the monotonicity of $\tilde I (2n-1)$, $n_c\neq 0$ iff $I(1)>0$. Then $I(1)=\int -V^{'}a\rho=\frac{\partial F}{\partial m}$ which is part (\textit{iii})
\end{proof}

The rest of this section is devoted to studying critical transitions in $\sigma$: whether and when the one solution regime turns to the three, and/or vice versa. 

This can be restated in terms of the number of roots of $F_\sigma(m)$, section \ref{sec:setup} and in the bistable case, this reduces by Proposition \ref{psd} (\textit{iii}) to studying the sign of $\left.\frac{\partial F}{\partial m}\right\vert_{m=0}(\sigma)$, i.e 

\begin{lemma}
	The critical transition(s) $\sigma_c$ corresponds to roots of $\left.\frac{\partial F}{\partial m}\right\vert_{m=0}(\sigma)$.
\end{lemma}

The forthcoming Proposition entirely characterises the possible number of roots and their existence.  For this the following assumptions are needed, and additionally for all remaining results of this section.

\begin{enumerate}
\item[(5)]\(\sup\limits_{0\leq x\leq x^*}\bar V(x,0)=\bar V(x^*,0)>0\)
\item[(6)]\(\inf\limits_{x\geq x^*}\bar V(x,0)=\bar V(x^*,0)\)
\end{enumerate}

\begin{remark}
	These assumptions generalise $\bar V(x,0)$ being strictly increasing
\end{remark}

\begin{proposition}\label{uniroot}(The Roots of \(\left.\frac{\partial F}{\partial m}\right\vert_{m=0}(\sigma)\))

	\(\left.\frac{\partial F}{\partial m}\right\vert_{m=0}(\sigma)\) can have at most one root. The exact number of roots can be characterised as follows:
	
	\begin{itemize}
	\item[$\centerdot$] There are no roots iff \(\lim\limits_{\sigma\downarrow 0}\left.\frac{\partial F}{\partial m}\right\vert_{m=0}(\sigma)<0\).
	
	\item[$\centerdot$] There is exactly one root iff \(\lim\limits_{\sigma\downarrow 0}\left.\frac{\partial F}{\partial m}\right\vert_{m=0}(\sigma)>0\)
	\end{itemize}
\end{proposition}

\begin{proof}
	
	Suppose there is a root of $\left.\frac{\partial F}{\partial m}\right\vert_{m=0}(\sigma)$ at $\sigma_c\in (0,\infty)$. Differentiating,
	\[
	\left.\frac{\partial^2 F}{\partial\sigma\partial m}\right\vert_{m=0}=-\left.\frac{2}{\sigma}\frac{\partial F}{\partial m}\right\vert_{m=0}+\frac{8}{\theta\sigma^5}\int-V^{'}a\bar V(x,0)\rho(x,0) dx	
	\]
	
	 At $\sigma_c$, the first term is 0 by assumption. For the second, set $\tilde V(x,0)=\frac{1}{k_{\bar V}}\bar V(x,0)$ with $k_{\bar V}=\bar V(x^*,0)>0$. Then as $0\leq\tilde V(x,0) \leq 1$, \(-V^{'}a\tilde V(x,0) < -V^{'}a\) so at $\sigma_c$
	\begin{equation}\label{critineq}
		\int \left. -V^{'}a\tilde V(x,0)\rho(\sigma_c,x,0) dx < \int-V^{'}a\rho(\sigma_c,x,0)=k\frac{\partial F}{\partial m}\right\vert_{m=0}(\sigma_c)=0
	\end{equation}
	where the strict inequality follows from $\rho>0$. By inequality (\ref{critineq}), the gradient at any root $\sigma_c$ must be negative, so if a root exists it must be unique. Reciprocally, \(\left.\frac{\partial F}{\partial m}\right\vert_{m=0}(\sigma)\leq0\) cannot have any roots and must be of constant sign.

	If \(\lim\limits_{\sigma\downarrow 0}\left.\frac{\partial F}{\partial m}\right\vert_{m=0}(\sigma)>0\) there must be a root as \(\lim\limits_{\sigma\uparrow \infty}\left.\frac{\partial F}{\partial m}\right\vert_{m=0}<0\). This can be seen by factoring the integral
	\begin{equation}\label{iqqq}
		\exp(-\frac{2}{\sigma^2}\mathcal{Z})\int \tilde a(-V^{'})\exp(-\frac{2}{\sigma^2}(\bar V(x,0)-\mathcal{Z}))
	\end{equation}	
	where $\mathcal{Z}=\min_x V(x,0)$ must exist by Assumption 1.

	In $[0,x^*]$ the integrand is positive and bounded, so can be bounded above by $\int_0^{x^*} a(-V^{'})dx$. Outside, it is negative and unbounded, so there exists some compact set $K\in[x^*,\infty)$ such that $\int_0^{x^*} a(-V^{'})dx<\int_K |a(-V^{'})|dx$. 
	
	Using $\lim\limits_{\sigma\uparrow\infty}\exp(-\frac{2}{\sigma^2}(\bar V-\mathcal{Z})\equiv 1$ and applying the dominated convergence theorem, integral (\ref{iqqq}) must eventually be negative. It follows that \(\lim\limits_{\sigma\downarrow 0}\left.\frac{\partial F}{\partial m}\right\vert_{m=0}(\sigma)>0\) there must be a root, and it is unique by the first claim in this Proposition.
\end{proof}

In other words if \(\lim\limits_{\sigma\downarrow 0}\left.\frac{\partial F}{\partial m}\right\vert_{m=0}(\sigma)<0\) the only stationary measure possible is the symmetric one. Otherwise there are three stationary measures below $\sigma_c$, and only the one (symmetric) above. The uniqueness of this point implies $\sigma_c^u=\sigma_c^u:=\sigma_c$. This rigidity in bifurcation pattern is striking and further investigated in \cite{alecio3}.

\begin{remark}\label{safety}
	If the limit \(\lim\limits_{\sigma\downarrow 0}\left.\frac{\partial F}{\partial m}\right\vert_{m=0}(\sigma)=0\) the behaviour will depend on whether the limit is approached from above or below.
	A very important example of this is if the maxima is at 0, as it coincides with a root of $a(-V^{'})$, the limit will be 0. However as $a(-V^{'})\geq 0$ in a neighbourhood of that root the limit will be approached from above, due of the assumed polynomial bounds on $-V^{'}$ and more careful estimates in the style of Theorem \ref{multiwell}. The result of Proposition \ref{uniroot} can then be applied, so we conclude there must be a root strictly greater than 0.
\end{remark}

\begin{remark}\label{safety1}
	An upcoming assumption, 7, implies $\rho$ must have a maximum in $[0,x^*)$ and, with Remark \ref{safety}, that \(\lim\limits_{\sigma\downarrow 0}\left.\frac{\partial F}{\partial m}\right\vert_{m=0}(\sigma)>0\) by the Laplace theorem, guaranteeing $\sigma_c$'s existence.
\end{remark}

The second part of the outlined programme - a study of the dependence of $\sigma_c$ on $\theta$ - is made possible by Proposition \ref{uniroot}. If $\sigma_c$ exists, it is unique, so the mapping between $\theta$ and $\sigma_c$ is a well defined function.

\begin{defn}[The Critical Transition function]\label{crticf}
	$\sigma^{*}:\theta\rightarrow\sigma_c$
	\begin{equation}
		\sigma^{*}(\theta)=
		\begin{cases}
			
			\sigma_c&\sigma_c \mathrm{ exists}  \\
			0&\mathrm{ otherwise} 
			\end{cases}
\end{equation}
\end{defn}


It was conjectured and numerically demonstrated in \cite{tug} that $\sigma_c$ exists $\theta>0$ and that $\sigma^{*}(\theta)$ is an increasing function, for $P^{'}=x$ and the simple polynomial bistable potential (Remark \ref{exs}). Using the results above, this conjecture can be proven under considerably weaker conditions.

This final result requires
\begin{enumerate}
	\item[(7)]\(\sup\limits_{0\leq x\leq x^*} \int^{x}\frac{P^{'}}{k^2}=\int^{x^*}\frac{P^{'}}{k^2}>0\)
	\item[(8)]\(\inf\limits_{x\geq x^*}\int^x\frac{P^{'}}{k^2}=\int^{x^*}\frac{P^{'}}{k^2}>0\)
\end{enumerate}
Additionally it is assumed that Assumptions 1-8 hold on $\theta\in J\subseteq\mathbb{R}^+$. Then 

\begin{proposition}\label{ctgb}
$\sigma^{*}(\theta)$ is an increasing function on $J$.
\end{proposition}

\begin{proof}
	
	Assumption 5 and 7 imply \(P^{'}(x^*)>0\) which further implies $-\bar V(x,0)$ has a maxima in $[0,x^*)$. As explained in Remark \ref{safety}, with the Laplace theorem, \(\lim\limits_{\sigma\downarrow 0}\left.\frac{\partial F}{\partial m}\right\vert_{m=0}(\sigma)>0\) (or 0 but approaching from above) and so $\sigma^*(\theta)>0$ (i.e a root of $\left.\frac{\partial F}{\partial m}\right\vert_{m=0}(\sigma)$) by Proposition \ref{uniroot}. 

Differentiating with respect to $\theta$,
\[
\left.\frac{\partial^2 F}{\partial\theta\partial m}\right\vert_{m=0}=-\left.\frac{1}{\theta}\frac{\partial F}{\partial m}\right\vert_{m=0}-\frac{4}{\theta\sigma^4}\int-V^{'}a\int^x\frac{P^{'}}{k^2}\rho(x,0) dx	
\]
With Assumptions 7 and 8, at $\sigma^*(\theta)=\sigma_c$

\begin{equation}\label{i2}
	-\left.\int -V^{'}a\tilde P\rho(\sigma^*,x,0) dx > - \int-V^{'}a\rho(\sigma^*,x,0)=c\frac{\partial F}{\partial m}\right\vert_{m=0}(\sigma_c)=0
\end{equation}
with $\tilde P = \int\frac{P^{'}}{k_Pk^2}$ with $k_P=\int^1\frac{P^{'}}{k^2}dx$. With identical reasoning to Proposition \ref{uniroot} it can be deduced that the derivative with respect to $\theta$ at any root $\sigma_c$ must be positive.

By the chain rule, 
\begin{equation}\label{ii4}
	\frac{d\sigma^*}{\partial\theta}(\theta)=-\frac{\frac{\partial^2 F}{\partial \theta\partial m}|_{m=0}}{\frac{\partial^2 F}{\partial \sigma\partial m}|_{m=0}}>0
\end{equation}
where inequality comes from inequalities (\ref{critineq}) and (\ref{i2}). So $\sigma_c(\theta)$ is an increasing function, and the claim proven.
\end{proof}

Returning to the very particular case considered by the conjecture of \cite{tug}, the maximum of $V_\theta$ lies in $[0,x^*]$ for $\theta>0$ and so with the reasoning of Remark \ref{safety}, $J=(0,\infty)$. 


\subsection{Multi-Well Potential}

This section is concerned with the natural generalisation of the symmetric double well potential: a symmetric potential which possesses multiple extrema in some finite interval. With more degrees of freedom, a greater multiplicity of behaviours can be exhibited. Rather than document this myriad of possibilities here, the primary interest is in finding criterion such that the symmetric multi-well case of self-consistency equation (\ref{sfcn}) behaves like the symmetric bistable case of section \ref{sec:bistable}.

Heuristically, it is reasonable to suppose that, for a multi-well potential $-V^{'}$ that is `minimally' negative in $[0,x^{*}]$, the self-consistency equation would exhibit `close to bistable' behaviour. The non-contiguousness of the positive regions of $-V^{'}$ present an immediate barrier to applying the results of Section \ref{sec:bistable} complicating this intuitive picture. 
This section will overcome these issues, developing attractively conjunct criteria on $-V^{'}$ to satisfy this programme, avoiding clumsy, overly prescriptive assumptions on the location of roots or extrema. 

\subsubsection{Definitions and Assumptions}
In light of this philosophy, the assumptions needed for this section are Assumptions 1-8 from Section \ref{sec:bistable}, lightly modified. As a corollary of Assumption 1, specifically the anti-symmetry of $-V^{'}$, there necessarily a root at 0. No further assumptions are demanded on the location or number of the roots.

Having been in essence already presented, the modifications are listed here
\begin{enumerate}
	\item[(2*)] $x^*$ is the root of $-V^{'}$ farthest from the origin, and $-V^{'}(x)<0$, $x>x^*$
	\item[(5*)]  \(\sup\limits_{0\leq x\leq x^*}\bar V_D(x,0)=\bar V_D(x^*,0)>0\) 
	\item[(6*)]\(\inf\limits_{x\geq x^*}\bar V_D(x,0)=\bar V_D(x^*,0)\)
	\end{enumerate}
where $\bar V_D(s,0)=\int^s_0\frac{V^{'}_D+P^{'}}{k^2}$ and $-V^{'}_D:=(-V^{'})_+ - \mathbb{I}_{[x^*,\infty)}(-V^{'})_{-}$\footnote{$(f)_-=\max(0,-f)$}

Assumption 2 has been loosened to admit multi-well potentials. Assumptions 5-6 have have been modified to apply to a dominating bistable potential $-V^{'}_D\geq-V^{'}$ for Corollary \ref{unboundedbelow}. When Assumption 2/5/6 is referred to in the following, it should be read as $2^*$/$5^*$/$6^*$. All the results in this section require Assumptions 1-8, unless otherwise stated. 

Intuitively, as $\sigma$ increases, the contribution to $F_\sigma$ of the integrand over $[x^*,\infty)$ will dominate that over $[0,x^*]$, just as in the bistable case. This first result established a threshold $\sigma_r$ above which $[x^*,\infty)$ will dominate the negative parts of the integrand in $[0,x^*]$, which is crucial to the sequel

\begin{proposition}\label{ssfn}
	$\int (\tilde{a}^3-\tilde{a})(-V^{'})_{-}\rho dx$ has exactly one root, $\sigma_r$. Moreover 
	\begin{equation}\label{eq23}\int_0^\infty (\tilde{a}^3-\tilde{a})(-V^{'})_{-}\rho dx > 0\end{equation} for $\sigma>\sigma_r$
\end{proposition}
\begin{proof}
	To highlight the similarities with Proposition \ref{uniroot}, define $H(\sigma):=-\int (\tilde{a}^3-\tilde{a})(-V^{'})_{-}\rho dx$. Now the proof runs identically: bound $\frac{dH}{d\sigma}$ above by $G$ using Assumptions 5 \& 6 to show the gradient at any root must be negative. Conclude there can only be at most one root, that exist iff $\lim\limits_{\sigma\downarrow 0}H(\sigma)>0$ 
	
	There are then two cases:
	\begin{itemize}
	\item[$\centerdot$] There are no roots iff $\lim\limits_{\sigma\downarrow 0}\int (\tilde{a}^3-\tilde{a})(-V^{'})_{-}\rho dx>0$, 
	
	\item[$\centerdot$] There is exactly one root iff $\lim\limits_{\sigma\downarrow 0}\int (\tilde{a}^3-\tilde{a})(-V^{'})_{-}\rho dx<0$
	\end{itemize}
	However there must be a maxima for $\rho$ in $[0,x^{*}]$ by Assumptions 7-8, so $\lim\limits_{\sigma\downarrow 0}\int (\tilde{a}^3-\tilde{a})(-V^{'})_{-}\rho dx\leq 0$ by the Laplace method. If strictly negative, the claim follows immediately. If the limit is 0, it will be approached from below by the same reasoning as Remark \ref{safety} and the defintion of $(-V^{'})_{-}$. Therefore $\sigma_r$ must exist.  

	That $\int (\tilde{a}^3-\tilde{a})(-V^{'})_{-}\rho dx > 0$ for $\sigma>\sigma_r$ follows from establishing the limit $\sigma\uparrow\infty$ is negative in the same manner as Proposition \ref{uniroot}.
\end{proof}
The importance of $\sigma_r$ becomes apparent in this next Proposition, as the point where the series expansion for $F_\sigma$ behaves much like Proposition \ref{psd}, with monotonically decreasing (scaled) coefficients.

\begin{proposition}\label{ordered}
For $\sigma>\sigma_r$, $F_\sigma(m)$ has at most one strictly positive (respectively negative) root of $F_\sigma$, which exist iff $\left.\frac{\partial F}{\partial m}\right\vert_{m=0}(\sigma)>0$  	
\end{proposition}
\begin{proof}
We will demonstrate that if $\int_0^\infty (\tilde{a}^3-\tilde{a})(-V^{'})_{-}\rho dx > 0$ then $\{\tilde I_\sigma(2n-1)\}_n$ is decreasing. The full result follows identically to Proposition \ref{psd}.

As the positive and negative regions of $(-V^{'})$ in $[0,x^*]$ are not contiguous, to easily represent them the following are introduced:
	\begin{align}\label{nct}
		\begin{split}
	a_n=\int_0^{x^*}\tilde a^{2n-1}(-V^{'})_{+}\rho dx\\
	b_n=\int_0^{x^*}\tilde a^{2n-1}(-V^{'})_{-}\rho dx\\
	c_n=\int^{\infty}_{x^*}\tilde a^{2n-1}(-V^{'})_{-}\rho dx
		\end{split}
	\end{align}
all of which are bounded below by 0.

In the terminology established in Proposition \ref{psd}, $-c_n=\tilde I_\sigma(2n-1,[x^*,\infty))$ and $a_n-b_n=\tilde I_\sigma(2n-1,[0,x^*])$

Then the condition $\int (\tilde{a}^3-\tilde{a})(-V^{'})_{-}\rho dx > 0$ is equivalent to
\begin{equation}\label{orderedineq}
(b_2-b_1)+(c_2-c_1)>0	
\end{equation}
the terms in the first bracket being negative increasing, the second positive increasing which can be seen by directly considering the integrands.
Therefore, $\forall n$

\begin{equation}\label{cond}
(b_n-b_{n-1})+(c_{n}-c_{n-1})>(b_2-b_1)+(c_2-c_1)>0	
\end{equation}

By identical reasoning to Proposition \ref{psd} we can represent self consistency function $F_\sigma$ by its power series (\ref{scps}). In terms of (\ref{nct}), the difference between any two rescaled coefficients is

\begin{equation}\label{sp1}
	\tilde I_\sigma(2n+1)-\tilde I_\sigma(2n-1)=a_{n+1}-a_n-\Big((b_{n+1}-b_n)+(c_{n+1}-c_n)\Big)
\end{equation}

To complete the proof of the claim, it suffices to show (\ref{sp1}) must be negative.
That $a_{n+1}-a_n$ is negative in was demonstrated in inequality (\ref{ineq}). 
It remains to be shown that the terms inside the parenthesis must be positive, but this was the content of inequality (\ref{cond}).
\end{proof}

The next corollary contains the same result as Proposition \ref{upb}, but with an entirely different proof to cover symmetric multi-well potentials 

\begin{cll}\label{unboundedbelow}
	There exists an upper critical threshold $\sigma_c^u$ such that, $\sigma>\sigma_c^u$, $F_\sigma(m)<0,\,m>0$.
\end{cll}

\begin{proof}
	By identical reasoning to Proposition \ref{psd}, $F_\sigma(m)$ can be represented by power series \ref{scps}. Then, for $\sigma>\sigma_r$ the coefficients $I_\sigma(n)$ are negative iff $I_\sigma(1)<0$. 
	
	Recalling the related dominating bistable potential 
	$$-V^{'}_D:=\mathbb{I}_{[0,x^*]}(-V^{'})_+ - \mathbb{I}_{[x^*,\infty)}(-V^{'})_{-}$$
	it can be seen that 
	\begin{equation}\label{ineq2}
		\tilde I_\sigma(1)=a_1-b_1-c_1<a_1+0-c_1=\int -V^{'}_D\tilde a\rho dx:=\tilde I^D_\sigma(1)
	\end{equation}	
	where the notation established in (\ref{nct}) has been used.

	$-V^{'}_D$ is a bistable potential satisfying all the salient original assumptions, so Proposition \ref{uniroot} is applicable, whence there exists some $\sigma_c^D$ such that  $\tilde I^D_\sigma(1)<0,\,\sigma>\sigma_c^D$
	Inequality (\ref{ineq2}) implies $\tilde I_\sigma(1)$ must be negative above the same threshold, so $\sigma_c^u$ must exist and 
	$$\sigma_c^u\leq\max(\sigma_c^D,\sigma_r)$$ \end{proof}

Given its existence, we set $\sigma_c^u$ to be the smallest $\sigma$ above which $F_\sigma(m)$ has no positive (negative respectively) roots
\begin{defn}\label{ctp}
	The Critical Transition function $\sigma_c(\theta)=\mathrm{inf}\{\sigma: F_{(s,\theta)}(m)<0,\,\forall m>0,\, \forall s>\sigma\}$ is well-defined.
\end{defn}
\begin{proof}This set is non-empty by Corollary \ref{uniroot} ($\sigma_c<\sigma_c^u$), and bounded below above 0.\end{proof}

Again, in terms of stationary measures of MV-SDE (\ref{proto}):
\begin{theorem}[Stationary Measures of MV-SDE (\ref{proto})]
	There exists an upper critical threshold $\sigma_c^u$ such that, $\sigma>\sigma_c^u$, MV-SDE (\ref{proto}) has only one stationary measure.
\end{theorem}

In the multi-well case the upper and lower critical thresholds will not necessarily be equal, and it is the upper threshold that is the true analogue of the critical threshold of Section \ref{sec:bistable}. The rest of this section is dedicated to a study of the dependence of the upper critical transition dependence on $\theta$. 

\begin{remark}Applying the results of the foregoing section to the dominating bistable process at most would show that the critical transition of the multi-well process is bounded above by an increasing function. \end{remark}

In section \ref{sec:bistable} this relied upon the bijection between critical points and roots of $\left.\frac{\partial F}{\partial m}\right\vert_{m=0}(\sigma)$, which in turn relies on ordered, decreasing coefficients $I(2n-1)$ of the power series expansion of $F_\sigma(m)$. This is no longer globally true, however if $I_{\sigma_r}(1)>0$ then $\sigma^u_c>\sigma_r$, where the coefficients are ordered and much of the machinery developed for the bistable potential can be repurposed. 

The next Proposition translates the result of \ref{sec:bistable} as directly as possible. As with Proposition \ref{ctgb}, let Assumptions 1-8 hold for $\theta\in J$.

\begin{proposition}
	\label{vainilla}
Suppose,
\begin{align}
\label{ii5}
\left.\frac{\partial F}{\partial m}\right\vert_{m=0}(\sigma_r)=\tilde I_{\sigma_r}(1)>0 
\end{align}
and
\begin{align}
	\label{ii1}\int_0^\infty \tilde a(-V^{'})(1 - \tilde V)\rho dx > 0\\ 
	\label{ii2} \int_0^\infty \tilde a(-V^{'})(1 - \tilde P)\rho dx > 0
\end{align}
for $\sigma\in\mathbb{R}^+$ and $\theta\in J\subseteq\mathbb{R}^+$

Then $\sigma^u_c>\sigma_r$ and $\sigma^u_c(\theta)$ is an increasing function on $J$.
\end{proposition}
\begin{proof}
	With (\ref{ii5}), $\sigma_c^u$ must necessarily be greater than $\sigma_r$
	Inequality (\ref{ii1})  implies (\ref{critineq}) so by the same process as Proposition \ref{uniroot}, we know $\tilde I_{\sigma} (1)$ has one root, $\sigma>\sigma_r$. By the monotonicity of the $\tilde I(2n-1)$ from Proposition \ref{ordered}, $\sigma_c$ will coincide with the root of $\tilde I_{\sigma} (1)$. 

	Inequality (\ref{ii2}) implies (\ref{i2}), so we conclude the result with the chain rule (\ref{ii4}), identically to Proposition \ref{ctgb}.
\end{proof}

By maintaining the utmost generality, the proposition is cumbersome in use. It requires the calculation of $\sigma_r$ rather than simply relying on its existence, while global inequalities (\ref{ii1}) and (\ref{ii2}) depend on $(\sigma,\,\theta)$ through $\rho$. Contrastingly, in the bistable case, inequalities (\ref{ii1}) and (\ref{ii2}) were implied by conditions directly on $\bar V^{'}$ and $P^{'}$, unmoderated by $\rho$. In the `close to bistable' regime with which we are concerned there is good reason to suppose conditions on $-V^{'},\,P^{'}$ alone  would be sufficient to imply (\ref{ii1}) and (\ref{ii2}), see the remarks that head this section.


As we will demonstrate, if $\bar V^{'}(x,0)$ and $P^{'}$ are strictly increasing, it is possible to find intuitive integral inequalities on $(-V^{'})_{+}$ versus $(-V^{'})_{-}$ at a small cost, see Remark \ref{remcost}. Further, a similar inequality can be derived on the potential to replace the need for the precise location of $\sigma_r$, with no further penalty. The following corollary exemplifies the above discussion. Although dealing specifically with the quadratic interaction it can be generalised to $P^{'}$ strictly increasing and $k$ such that $\lim_{x\uparrow\infty}\tilde a(x) >\sqrt{2}$ for $\theta\in J$ such that $\bar V^{''}(x,0)>0$ with no change in argumentation. 

\begin{cll}[Multi-Well Potential with Quadratic Interaction]$ $\\
	\label{c2fg}
	Consider SDE (\ref{proto}) with $P^{'}=x$, $k=1$ and $-V^{'}$ a multi-well potential satisfying Assumptions 1-4. Suppose $-V^{'}$ additionally satisfies Inequalities 
	\begin{align}\label{c2fg1}
		\int_0^t x(1-x)(-V^{'})_{+}-x(-V^{'})_{-} dx > 0,&\quad\forall t < x^{*} \\\label{c2fg2}
		\int_0^t x\big((-V^{'})_{+}-2(-V^{'})_{-}\big) dx > 0,&\quad\forall t<\sqrt{2}x^{*}
	\end{align}
	 Then there exists $\theta^*$ such that the upper critical temperature function is decreasing on $[\theta^*,\infty)$ 
\end{cll}

\begin{proof} 
	As $P^{''}$ is bounded below, by Lemma \ref{eventmono}, above some $\theta^*$, $\bar V^{''}(x,0)>0$ . This implies Assumptions 5-6 and further implies $P^{'}$ strictly increasing, which in turn implies Assumptions 7-8 on $[x^*,\infty)$.

	To reduce the need for tildes, we rescale the above inequalities in order to be able to take $x^{*}=1$.
	For inequality (\ref{ii1}) it is sufficient to require
	\begin{equation}\label{i1}\int_0^1 x(-V^{'})(1 - \tilde V)\rho dx > 0,\, \forall\sigma>0 \end{equation} as the contribution to the integral over $[1,\infty)$ is positive by Assumption 2.\footnote{If needed, a less restrictive inequality can be derived by bounding below the integral in this outer region using inequality (\ref{orderedineq}).} For this inequality to hold $(-V^{'})_{+}$ must `dominate' $(-V^{'})_{-}$ in $[0,1]$, which is intuituvely the case if $V$ is `close to bistable'. 
	
	We can further simplify this inequality further using the convexity of $\tilde V$, guaranteed by the second derivative test. Given $\tilde V(0) = 0$ and $\tilde V(1) = 1$ by definition, $1-x<1-\tilde V< 1$. Consequently (\ref{i1}) can be bounded below with 
	\[
		\int_0^1 x\big((-V^{'})_{+}-(-V^{'})_{-}\big)(1 - \tilde V)\rho dx > \int_0^1 x(1-x)(-V^{'})_{+}\rho dx - \int_0^1 x(-V^{'})_{-}\rho dx
	\] 
	so it is suffices that 
	\begin{equation}\label{i4}
		\int_0^1 x(1-x)(-V^{'})_{+}\rho dx > \int_0^1 x(-V^{'})_{-}\rho dx
	\end{equation}
	for all $\sigma\in\mathbb{R}^+$. 

	In general an argument like that just performed would be required to simplify (\ref{i2}). However with this choice of $P^{'}$ and $k$, (\ref{i2}) such a manipulation yields (\ref{eq23}), which is of course true for $\sigma>\sigma_r$ by Proposition \ref{ssfn}.

	Finally, to ensure $\sigma_r<\sigma_c$, it suffices to find conditions that imply $I_\sigma^*(1)>0$. 
	Contrastingly, an inequality on the sign $I_\sigma(1)$ cannot be global in $\sigma$, because it must be unbounded below as $\sigma\uparrow\infty$, as expatiated in Corollary \ref{unboundedbelow}. 

	For $\sigma\leq\sigma_r$, 
	\begin{equation}\label{e1}
		\int_0^1(x-x^3)(-V^{'})_{-}\rho dx \geq \int_1^\infty(x^3-x)(-V^{'})_{-}\rho dx
	\end{equation}
	by Proposition \ref{ssfn}. Then
	\[ \int_1^\infty x(-V^{'})_{-}\rho dx < \int_1^{\sqrt{2}} x(-V^{'})_{-}\rho dx + \int_0^1(x-x^3)(-V^{'})_{-}\rho dx < \int_0^{\sqrt{2}} x(-V^{'})_{-}\rho dx
	\]
	where the first inequality comes from simple bounding of the polynomial terms, while the second from inequality (\ref{e1}).

	Inserting this into $\tilde I_\sigma(1)$,
	\begin{align}
		\tilde I_\sigma(1) = \int_0^1 x(-V^{'})_{+}\rho dx - \int_0^1 x(-V^{'})_{-}\rho dx - \int_1^\infty x(-V^{'})_{-}\rho dx \\ > \int_0^1 x(-V^{'})_{+}\rho dx - 2\int_0^{\sqrt{2}} x(-V^{'})_{-}\rho dx 
	\end{align}
	So, if $\forall\sigma\in\mathbb{R}^+$ 
	\begin{equation}\label{bigin}\int_0^1 x(-V^{'})_{+}\rho dx > 2\int_0^{\sqrt{2}} x(-V^{'})_{-}\rho dx\end{equation}
	then $\tilde I_\sigma(1)>0$ when $\sigma\leq\sigma^*$

	As $\rho$ is strictly decreasing, inequality conditions (\ref{i4}) and (\ref{bigin}) are implied by 
	\begin{align}
		\label{ii20}\int_0^t x(1-x)(-V^{'})_{+}-x(-V^{'})_{-} dx > 0,&\quad\forall t < 1 \\
		\label{ii10}\int_0^t x\big((-V^{'})_{+}-2(-V^{'})_{-}\big) dx > 0,&\quad\forall t<\sqrt{2}
	\end{align}
	by the second mean value theorem theorem for integrals.

	Therefore the assumptions of this corollary imply those of Proposition \ref{vainilla}, which proves the result.
\end{proof}

\begin{remark}\label{remcost}
	With strictly increasing $\bar V(x,0)$, $\rho$ must have global maxima at 0. It is possible to chose $\epsilon$ such that $x(-V^{'})\geq 0$ or $\leq 0$ on $[-\epsilon,\epsilon]$. The limit of integrals (\ref{ii1}) and (\ref{ii2}) as $\sigma\downarrow 0$ is, applying the Laplace method, 0 and the inequalities imply the limit is approached from above. So, by Remark \ref{safety}, $-xV^{'}\geq 0$ for arbitrarily small $x$ implying $-V^{'}(x)>0$ necessarily, for some small interval $0<x<\epsilon$  
	
	This can be remedied by finding some new lower bound for $\sigma$ that those inequalities hold. However that would prevent further simplification by application of the second mean value theorem for integrals.

\end{remark}

These assumptions are flexible enough to admit a potential with any number of extrema. Indeed, we demonstrate this by outlining a procedure to construct an admissible $-V^{'}$ with an arbitrary number of roots. Perhaps the most direct method is to start with a multi-well potential with the required number of roots that is positive in $[0,x_1]$ and $[x_2,x^*]$ and satisfies Assumptions 1 to 4. It is possible to find coefficients $\alpha_i$\footnote{Assumptions 1-4 are not affected by such a scaling, although $\theta^{*}$ will increase} such that $\alpha_1\int_0^{x_1}x(1-x)(-V^{'})\rho dx>\int_{x_1}^{x_2}x(-V^{'})_{-}\rho dx$ and $\alpha_2\int_{x_2}^{x^*}x(-V^{'})\rho dx>2\int_{x^*}^{\sqrt{2}}x(-V^{'})_{-}\rho dx$. Then for $\alpha_1\mathbb{I}_{[0,x_1]}(-V^{'})+\mathbb{I}_{[x_1,x_2]\cup[1,\infty)}(-V^{'})+\alpha_2\mathbb{I}_{[0,x_1]}(-V^{'})$ satisfies inequalities (\ref{c2fg1} \& \ref{c2fg2}).

\section{Conclusion}

In this work, we have studied the possible phases and their transition points for MV-SDE (\ref{proto})
We have shown for sufficiently small $\sigma$ there are exactly as many stationary measures as roots of $V^{'}$ and sufficiently large there is only one. In the case of symmetrical potentials we have gone further and additionally demonstrated the (upper) critical transition is a strictly increasing function of the aggregation parameter.

The approach utilised is direct, relying upon the first MEE, and robust enough to generate quantitative estimates.
In addition to entirely novel results, where similar results have presented before, their proofs have been simplified and their applicability greatly increased.
A choice was made to keep assumptions as general as possible, and the results can yield more when more is known of $V^{'},\,P^{'}$

For related future work, an idea that immediately presents itself is whether this machinery can be brought to bear on MV-SDE (\ref{proto}) with coloured noise. \cite{vaes} studied phase transitions using a small parameter expansion approach, to which our methodology can be employed to fully understand the individual correction terms.

\bibliographystyle{abbrv}
\bibliography{ggbiblio} 

\end{document}